\documentclass[10 pt]{amsart}

\usepackage{amsmath} 
\usepackage{amsfonts}
\usepackage{amssymb}
\usepackage{amstext}
\usepackage{amsbsy}
\usepackage{amsthm}
\usepackage{mathtools}
\usepackage{color}
\usepackage{hyperref}
\usepackage{enumitem}

\newtheorem{theorem}{Theorem}[section]
\newtheorem{question}[theorem]{Question}

\newtheorem{lemma}[theorem]{Lemma}
\newtheorem{prop}[theorem]{Proposition}

\newtheorem{example}[theorem]{Example}
\newtheorem*{conj1'}{Conjecture 1'}

\theoremstyle{definition}

\newtheorem{defn}[theorem]{Definition}
\newtheorem{remark}[theorem]{Remark}

\newcommand{\R}{\mathbb{R}}
\newcommand{\Z}{{\mathbb Z}}

\newcommand{\N}{{\mathbb N}}

\newcommand{\vre}{\varepsilon}

\newcommand{\proj}{\operatorname{proj}}

\newcommand{\spn}{\operatorname{span}}

\newcommand{\im}{\operatorname{im}}
\newcommand{\orb}{\operatorname{orb}}

\newcommand{\ignore}[1] {}
\newcommand{\df}{{\, \stackrel{\mathrm{def}}{=}\, }}

\newcommand{\rst}[1]{\ensuremath{{\mathbin\upharpoonright}%
\raise-.5ex\hbox{$#1$}}}

\title{Finite Orbits in Random Subshifts of Finite Type}

\author{Ryan Broderick}
\address{340 Rowland Hall (Bldg. \#400), University of California, Irvine, Irvine, CA 92697-3875, USA}
\email{broderir@math.uci.edu}

\begin{document}
\begin{abstract} 
For each $n, d \in \N$ 
and $0 < \alpha < 1$, we define a random subset of $\mathcal{A}^{\{1, 2, \dots, n\}^d}$
by independently including each element with probability $\alpha$
and excluding it with probability $1-\alpha$, and consider the associated random subshift of finite type.
Extending results of McGoff \cite{M} and of McGoff and Pavlov \cite{MP}, we prove
there exists $\alpha_0 = \alpha(d, |\mathcal{A}|) > 0$ such
that for $\alpha < \alpha_0$ and with probability tending to $1$ as $n \to \infty$, 
this random subshift will contain only finitely many elements.
In the case $d = 1$, we obtain the best possible such $\alpha_0$, $1/|\mathcal{A}|$.
\end{abstract} 

\maketitle

\section{introduction}

Fix a finite color set $\mathcal{A}$ with at least two elements.
Let $\Omega_n^d$ be the power set of $\mathcal{A}^{\{1,2,\dots, n\}^d}$.
Given $\omega \in \Omega_n^d$, let $X_\omega$ be the set of colorings $\eta \in \mathcal{A}^{\Z^d}$
such that for every $\mathbf{j} \in \Z^d$ there exists $\beta \in \omega$ such that
$\eta(\mathbf{i} + \mathbf{j}) = \beta(\mathbf{i})$ for all $\mathbf{i} \in \{1,2,\dots, n\}^d$.
That is, $X_\omega$ is the $\Z^d$-subshift of finite type that has $\omega$ as its set of allowed patterns.
For convenience, we will abuse notation slightly and, when 
$\gamma_1 \colon D \to \mathcal{A}$ and $\gamma_2 \colon D + \mathbf{j} \to \mathcal{A}$
satisfy $\gamma_1(\mathbf{i}) = \gamma_2(\mathbf{i} + \mathbf{j})$ for all $\mathbf{i} \in D$,
we will say $\gamma_1 = \gamma_2$.
Thus, $\eta \in X_\omega$ if and only if each of its restrictions to a subset 
of the form $\mathbf{j} + \{1,2,\dots, n\}^d$ is equal to an element of $\omega$.
We will also say that a coloring $\gamma \colon B \to \mathcal{A}$ (where $B \subset \Z^d$)
is \emph{$\omega$-legal} if for each $\mathbf{j} \in \Z^d$ such that $\{1,2,\dots, n\}^d + \mathbf{j} \subset B$,
$\gamma|_{\{1,2, \dots, n\}^d + \mathbf{j}} \in \omega$. Note that this is weaker than requiring
that $\gamma = \eta|_{B}$ for some $\eta \in X_\omega$, as it is possible that 
$\gamma$ is $\omega$-legal but cannot be extended to an $\omega$-legal coloring of $\Z^d$.

Note that if $\omega = \mathcal{A}^{\{1,2,\dots, n\}^d}$, then $X_\omega$ is the full shift $\mathcal{A}^{\Z^d}$.
If $\omega = \varnothing$, then $X_\omega$ is empty, but it also possible for $X_\omega$ to be empty
when $\omega$ is nonempty. For example, let $d = 1$, $n = 2$, $\mathcal{A} = \{0, 1\}$, and 
$\omega = \{\beta\}$, where $\beta(1) = 0$ and $\beta(2) = 1$. Then if there exists $\eta \in X_\omega$,
$\eta(1) \neq 1$, since otherwise $\eta|_{\{1,2\}} \neq \beta$, and $\eta(1) \neq 0$, since otherwise
$\eta|_{\{0,1\}} \neq \beta$. But $\eta(1) \in \mathcal{A} = \{0,1\}$, so we have a contradiction
and $X_\omega$ is empty.

Put the usual topology on the full $\Z^d$-shift via the metric
$$\rho(\eta_1, \eta_2) = 2^{-\min\{\|\mathbf{m}\| \colon \eta_1(\mathbf{m}) \neq \eta_2(\mathbf{m})\}}.$$
Then put the subspace topology on each subshift $X_\omega$. Thus, together
with the translations $T^{\mathbf{n}} \colon X_\omega \to X_\omega$ given
by $T^{\mathbf{n}}(\eta)(\mathbf{i}) = \eta(\mathbf{i} + \mathbf{n})$,
$X_\omega$ is a $\Z^d$-topological dynamical system and we may study dynamical properties
such as periodicity, entropy, and directional entropy.

We will study these properties for \emph{random} subshifts in the following sense.
For each $d, n \in \N$ and $0 < \alpha < 1$, define a probability measure
$\mu_{\alpha, n}$ on $\Omega_n^d$ by independently
including each element of $\mathcal{A}^{\{1,2, \dots, n\}^d}$ with probability $\alpha$,
and excluding it with probability $1-\alpha$. Thus, for any $\omega \in \Omega_n^d$,
$\mu_{\alpha, n}(\{\omega\}) = \alpha^{|\omega|}(1-\alpha)^{|\mathcal{A}|^{n^d}-|\omega|}$.
Here and throughout the paper, $|\cdot|$ is used to denote the cardinality of a finite set.
Note that $\mu_{\alpha, n}$ also depends on $d$ and $\mathcal{A}$, but we suppress these
in order to keep the notation manageable.

Notice that the number of $\omega$-legal colorings of $\{1, 2, \dots, n\}^d$ has binomial distribution with parameters $|\mathcal{A}|^{n^d}$
and $\alpha$. Thus, using Chernoff's inequality, it is easy to see that for any $\alpha \in (0,1)$ and any $0 < \beta < \alpha$,
there exists $c > 0$ such that the number of legal colorings of $\{1, 2, \dots, n\}^{d}$ is at least $\beta|\mathcal{A}|^{n^d}$
with probability greater than $1 - e^{-c|\mathcal{A}|^{n^d}}$.
One might expect that the abundance of legal blocks leads to positive entropy, but
it was shown by K.\ McGoff in \cite{M} 
that if $d=1$, then $\frac{1}{|\mathcal{A}|}$ is the critical value for $\alpha$ in the sense
that the probability of $X_\omega$ having positive entropy tends to $0$ as $n$ tends to $\infty$ if $\alpha < 1/|\mathcal{A}|$ and tends to $1$ if $\alpha > 1/|\mathcal{A}|$. K.\ McGoff and R.\ Pavlov proved
a weaker version of this result for $d > 1$, which implies that if $\alpha \le 1/|\mathcal{A}|$,
then for every $\vre > 0$ the probability that $X_\omega$ has entropy at least $\vre$ tends to $0$.
So if $\alpha$ is small and $n$ is large, then the entropy of $X_\omega$ is, with very high probability,
a small positive number or zero. Their result does not distinguish between these two possibilities though
so it is natural to ask about the probability that the entropy is positive in the case $d > 1$.
Furthermore, for all $d \in \N$ it is possible for $X_\omega$ to have zero entropy, while still
containing aperiodic elements\footnote{For example, let 
$\omega = \{\beta_1, \beta_2, \beta_3\} \in \Omega_2^1$
where $\beta_1(1) = 0, \beta_1(2) = 0, \beta_2(1) = 0, \beta_2(2) = 1, \beta_3(1) = 1, \beta_3(2) = 1$.}.
(Here and throughout, we say a coloring $\eta \colon \Z^d \to \mathcal{A}$ 
is \emph{aperiodic} if it has no period vectors, i.e. there does not exist $\mathbf{p} \in \Z^d \setminus \{\mathbf{0}\}$
such that $\eta(\mathbf{j} + \mathbf{p}) = \eta(\mathbf{j})$ for all $\mathbf{j} \in \mathbf{Z}^d$.)
 Thus, it is also natural to investigate the likelihood that $X_\omega$ has at least one period vector, as well as the likelihood of
the stronger condition that $|X_\omega|$ is finite. 
We are able to 
determine both the entropy and finiteness 
for small enough $\alpha$ by showing the following.
\begin{theorem}
\label{aperiodic}
For any $d, |\mathcal{A}| \in \N$, there exists $\alpha_0 = \alpha_0(d, |\mathcal{A}|) > 0$ such that for $0 < \alpha < \alpha_0$, 
$$\lim_{n\to \infty} \mu_{\alpha,n}(|X_\omega| < \infty) = 1.$$
In the case $d = 1$, we can take $\alpha_0(1, |\mathcal{A}|) = \frac{1}{|\mathcal{A}|}$.
\end{theorem}

\begin{remark}
\label{bounded-period=finite}
Clearly, $X_\omega$ is finite if and only if there exists $p \in \N$ such that
every $\eta \in X_\omega$ is periodic in each cardinal direction with period less than $p$.
\end{remark}

\noindent We will prove that this latter property holds with probability tending to $1$ when $0 < \alpha < \alpha_0$.

Since the entropy of a finite subshift is clearly zero, the aforementioned result of McGoff implies that
the $\alpha_0$ in Theorem~\ref{aperiodic} is optimal in the case $d=1$.
For $d > 1$, our $\alpha_0$ will be below the critical value of $1/|\mathcal{A}|$, which leaves a 
gap in the parameter space. This leads to several open questions, which we discuss in \S\ref{entropy}.

The paper is organized as follows: In \S\ref{dimone} we prove the $d=1$ case of
Theorem~\ref{aperiodic}. The argument in this case is different from the one for arbitrary dimension
and provides a larger $\alpha_0$ than the general argument allows us to obtain.
In \S\ref{higherdim} we present the proof of Theorem~\ref{aperiodic} for $d\in \N$ arbitrary.
Finally, in \S\ref{entropy} we discuss further directions and open questions.

\section{The case $d =1$}
\label{dimone}

To prove the theorem in the case $d=1$ (and obtain the optimal value of $\alpha_0$ in this case),
we will use the above-mentioned result that for $\alpha < |\mathcal{A}|^{-1}$,
the probability that $X_\omega$ has positive entropy tends to $0$ as $n \to \infty$.

In our context, topological entropy can be defined 
using the notion of complexity.
Given a finite set $K \subset \Z^d$, define the complexity of $K$ with respect to $\omega$
to be $P_\omega(K) \df |\{ \eta |_{K} \colon \eta \in X_\omega\}|$.
In the case that $K = \Z^d \cap ([1, k_1] \times [1,k_2] \times \dots \times [1,k_d])$,
we will write $P_\omega(K) = P_\omega(k_1, k_2, \dots, k_d)$
and in the case $K = \Z^d \cap [1,k]^d$, we will write $P_\omega(K) = P_\omega(k)$.

\begin{defn}
If $\omega \in \Omega_n^d$, the (topological) \emph{entropy} of the system $X_\omega$ is
$$h(X_\omega) = \lim_{k \to \infty} \dfrac{\log(P_\omega(k))}{k^d}.$$
\end{defn}

\noindent It is straightforward to check that this definition coincides with the general definition
of topological entropy when $X_\omega$ is endowed with a topology as described in the introduction.

\begin{proof}[Proof of Theorem \ref{aperiodic} for $d = 1$]
Let $d = 1$ and $\alpha < \frac{1}{|\mathcal{A}|}$.
By remark \ref{bounded-period=finite}, it suffices to show that, with probability tending to $1$,
every element of $X_\omega$ is periodic with period at most $n$.
We prove this in two steps: First we show that $X_\omega$ will, with probability tending to $1$, 
contain no periodic colorings with period greater than $n-1$; then we show that the existence of an aperiodic coloring
would imply, with probability tending to $1$, the existence of a periodic coloring of large period, completing the proof.
Let $\omega \in \Omega_n^1$ and suppose $\eta \in X_\omega$ is periodic with minimal period $p \ge n$.
We first claim that there exist $\ell \ge n$ and $j_0 \in \Z$ such that
\begin{enumerate}
	\item $\eta(j_0 + j + \ell) = \eta(j_0 + j)$ for each $1 \le j \le n$\label{mingap1}
	\item For any $1 \le \ell' < \ell$ there exists $1 \le j \le n$ such that $\eta(j_0 + j + \ell') \neq \eta(j_0 + j )$.\label{mingap2}
\end{enumerate}
To see this, for each $j_0 \in \Z$ let $p(j_0) \le p$ be the smallest natural number such that
$\eta(j_0 + j + p(j_0)) = \eta(j_0 + j)$ for each $1 \le j \le n$. Choose $j_0 \in \Z $ that maximizes $p(j_0)$
and assume for contradiciton that $p(j_0) < n$.
If $\eta$ is periodic with period $p(j_0)$, then $p(j_0) = p \ge n$ and we are done.
Otherwise, there exists $j_1 \neq j_0$ such that $p(j_1) < p(j_0)$.
Without loss of generality we may assume $j_1 > j_0$ and, by renaming $j_0$ if necessary,
that $j_1 = j_0 + 1$. Let $p_0 = p(j_0)$ and $p_1 = p(j_1) = p(j_0 + 1)$. 
Then for $1 \le j \le n$, we have $\eta(j_0 + j) = \eta(j_0 + p_0 + j)$ 
and $\eta(j_0 + 1 + j) = \eta(j_0 + 1 + p_1 + j)$. Hence, $\eta(j_0 + j) = \eta(j_0 + p_1 + j)$ for $2 \le j \le n$.
Furthermore, since $p_1 < p_0 < n$, we have 
$$\eta(j_0 + 1 + p_1) = \eta(j_0 + 1 + p_0 + p_1) = \eta(j_0 + 1 + p_0) = \eta(j_0 + 1),$$
and hence $\eta(j_0 + j) = \eta(j_0 + j + p_1)$ for $1 \le j \le n$, contradicting the minimality of $p(j_0)$.
Thus, there exists $j_0 \in \Z$ satisfying conditions (\ref{mingap1}) and (\ref{mingap2}) above.

Now we split into two cases, and show that either possibility has probability approaching $0$ as $n \to \infty$.
\begin{enumerate}
\item[{\bf Case 1}:] \emph{There are no integers $j_0 \le j_1 < j_2 \le j_0 + \ell-1$ such that 
$\eta(j_1 + j) = \eta(j_2 + j)$ for each $1 \le j \le n$.}

In this case, we have that there exists an $\omega$-legal coloring $\alpha$ of $\{1, 2, \dots, \ell + n\}$
such that the colorings $\alpha|_{\{j+1, \dots, j+n\}}$ ($j = 0, 1, \dots, \ell-1$) are all distinct
and $\alpha(j) = \alpha(\ell+j)$ for each $1 \le j \le n$.
But for each $\ell$ there are at most $|\mathcal{A}|^{\ell}$ such colorings, each with probability at most 
$\alpha^{\ell}$, so the probability that such a coloring exists for some $\ell\ge n$ is at most 
$$\sum_{\ell = n}^\infty \alpha^{\ell}|\mathcal{A}|^{\ell} 
				= \dfrac{(\alpha|\mathcal{A}|)^{n}}{1-\alpha|\mathcal{A}|},$$
which tends to $0$.
\item[{\bf Case 2}:] \emph{There exist $j_0 \le j_1 < j_2 \le j_0 + \ell - 1$
						such that $\eta(j_1 + j) = \eta(j_2 + j)$ for each $1 \le j \le n$.}
						
In this case, we claim that $X_\omega$ must have positive entropy.
To see this, let $\beta = \eta|_{\{j_0 + 1,j_0 +  2, \dots, j_0 + \ell + n\}}$, and define 
$$\gamma : \{j_0 + 1, j_0 + 2, \dots, j_0 + \ell + n - (j_2 - j_1)\} \to \mathcal{A}$$
via $\gamma(j) = \eta(j)$ for $j_0 + 1 \le j \le j_1+n$ and $\gamma(j) = \eta(j + (j_2-j_1))$ for 
$j_1+n < j \le \ell + n - (j_2 - j_1)$.
Then $\gamma$ is also $\omega$-legal, and furthermore any coloring ``stitched together'' out of translates of $\beta$ and $\gamma$
is also $\omega$-legal. That is, given any sequence $\mathbf{k} = (k_i)_{i \in \Z}$ with the property that
for all $i \in \Z$, $k_{i+1} - k_i = g_i \in \{\ell, \ell - (j_2 - j_1)\}$, we may define an $\omega$-legal coloring
$\eta_{\mathbf{k}} : \Z \to \mathcal{A}$ in the following way: For $k_i < j \le k_{i+1}$,
let $\eta_{\mathbf{k}}(j) = \beta(j - k_i)$ if $g_i = \ell$ and let $\eta_{\mathbf{k}}(j) = \gamma(j - k_i)$ if $g_i = \ell - (j_2 - j_1)$.
Note that by our choice of $j_0$ and $\ell$, for each $L > \ell$ 
the set of $j_3 \in \{k_0 + 1, \dots, k_0 + L\}$
such that $\eta|_{\mathbf{k}}(j_3 + j) = \eta(j_0 + j)$ for $j \in \{1, 2, \dots, n\}$
is exactly $I(\mathbf{k}, L) = \{k_i \colon k_0 < k_i \le k_0 + L \}$.
Hence, if $\mathbf{k}$ and $\mathbf{k}'$ are two sequences as above which also satisfy
$k_0 = k_0' = 0$, then
$\eta|_{\mathbf{k}}$ and $\eta_{\mathbf{k}'}$ have distinct
restrictions to $\{1, \dots, L\}$ whenever $I(\mathbf{k}, L) \neq I(\mathbf{k}', L)$.
Since there are at least $2^{\lfloor L/\ell \rfloor}$
finite sequences $0 = k_0 < k_1 <  \dots < k_m$ which satisfy $k_{i+1} - k_i \in \{\ell, \ell - (j_2 - j_1)\}$
and $\sum_{i=1}^m (k_{i+1} - k_i) \le L$,
it follows that $|\{\eta|_{\{1, 2, \dots, L\}} \colon \eta \in X_\omega \}| \ge 2^{\lfloor L/\ell \rfloor}$,
so the entropy of $X_\omega$ is positive. By \cite{M}, the probability of this event tends to $0$.
\end{enumerate}

Hence, with probability tending to $1$, $X_\omega$ does not contain any periodic colorings of period at least $n$.
But if $X_\omega$ contains an aperiodic coloring $\eta$ then there must exist $j_0 \in \Z$
such that $\inf\{j_1 \colon \eta(j_0 + j) = \eta(j_1 + j)\text{ for each }1 \le j \le n\} \ge j_0 + n$.
If this infimum is a finite integer $j_1$, 
then we obtain a periodic coloring in $X_\omega$ with period at least $n$
by setting $\eta'((j_1-j_0)m + j) = \eta(j_0 + j)$ for each $m \in \Z$ and $1 \le j \le (j_1 - j_0)$,
so by the above we may assume the infimum is infinite for each such $j_0$. 
Since there are only finitely many colorings of $\{1, 2, \dots, n\}$,
there is a maximal such $j_0$, which implies that $\eta|_{\{j_0 + 1, \dots\}}$ is periodic. 
Since this induces a periodic element of $X_\omega$,
we may assume the minimal period $p$ is strictly less than $n$, by the above. 
Thus, there exists an $\omega$-legal coloring $\gamma$ of $\{1, 2, \dots, 2n\}$
such that
\begin{enumerate}
\item $\gamma(j) = \gamma(j + p)$ for $n- p < j \le 2n - p$
\item $\gamma(n-p) \neq \gamma(n)$
\end{enumerate}
Note that such a coloring is determined by $p$ and by its restriction to $\{1, 2, \dots, n\}$,
so there are at most $n |\mathcal{A}|^n$ of them.
We claim that the colorings $\beta_j : \{1, 2, \dots, n\} \to \mathcal{A}$
given by $\beta_j(i) = \gamma(j+i)$ for $0 \le j \le n-1$ are all distinct.
To see this, let $m(\beta) = \max\{i \colon \beta(i) \neq \beta(i + p)\}$.
For $0 \le j \le n - p-1$, $m(\beta_j) = n - j - p$, so these are all distinct.
For $n-p \le j \le n-1$, $\{i \colon \beta_j(i) \neq \beta_j(i+p)\}$ is the empty set,
so $\beta_j$ is distinct from $\beta_{j'}$ for any $0 \le j \le n-p-1 < j' \le n - 1$.
We claim that, if $n-p \le j < j' \le n-1$, then
$\beta_j$ and $\beta_{j'}$ are distinct as well.
Indeed, suppose there exist $n-p \le j < j' \le n-1$
with $\beta_j = \beta_{j'}$ and let $p' = j' - j < p$.
Then for $j +1 \le i \le j + n$, $\gamma(i) = \gamma(i + p')$.
Now, let $i > n- p$ be arbitrary. Since $p < n$ there exists $k \in \Z$ such that 
$i + kp \in [j+1, j+ n]$ and so 
$$\gamma(i) = \gamma(i + kp) = \gamma(i + kp + p') = \gamma(i + p').$$
Since $i > n-p$ was arbitrary, this violates the minimality of the period $p$.
Hence, the $\beta_j$ are all distinct for $0 \le j \le n-1$,
so the probability of such a coloring of $\{1, 2, \dots, 2n\}$ existing is
at most $n |\mathcal{A}|^n\alpha^n$. Since this tends to zero, the probability that $X_\omega$ contains an aperiodic coloring does as well.
\end{proof}

\begin{remark}
As in \cite{M}, we may analogously define a random subshift of a fixed SFT $X$. 
If the entropy of $X$ is $\log \lambda$,
then for large $\ell$ the number of words of length $\ell$ that appear in $X$ is approximately $\lambda^\ell$.
It is proved in \cite{M} that if $\alpha < 1/\lambda$, then the probability that $X_\omega$ has positive entropy tends to $0$ as $n \to \infty$.
Using these two facts and
making the obvious changes to the above proof, we
see that for $\alpha < 1/\lambda$, the probability that a random SFT $X_\omega \subset X$ is finite tends to $1$ as $n \to \infty$.
And again by \cite{M}, this value $1/\lambda$ is optimal.
\end{remark}

\section{Higher dimensions}
\label{higherdim}

\ignore{
\begin{lemma}
$\eta \in X_\omega$ has finite orbit if and only if there exist 
linearly independent integer vectors $\mathbf{p}_1, \dots, \mathbf{p}_d$
such that $\eta(\mathbf{j} + \mathbf{p}_i) = \eta(\mathbf{j})$ for every 
$\mathbf{j} \in \Z^d$ and every $1 \le i \le d$.
\end{lemma}

\begin{proof}
First suppose $\{T^{\mathbf{m}} \eta \colon \mathbf{m} \in \Z^d\}$ is a finite set.
Then in particular $\{T^{m\mathbf{e}_1} \eta \colon m \in \Z\}$
is finite. Thus, there exists a nonzero $m \in \Z$
such that $T^{m\mathbf{e}_1}\eta = \eta$, i.e. 
$\eta(\mathbf{j} + m\mathbf{e}_1) = \eta(\mathbf{j})$ for every $\mathbf{j} \in \Z^d$.
Arguing similarly for the other $d-1$ standard basis vectors we obtain one direction of the lemma.

Now suppose there are linearly independent period vectors
$\mathbf{p}_1, \dots, \mathbf{p}_d$. Clearly there exist $m_i \in \Z$ such that, for each $1 \le i \le d$,
$m_i \mathbf{e}_i = \displaystyle\sum_{j=1}^d a_{i,j} \mathbf{p}_j$ for some $a_{i,j} \in \Z$.
It follows that the vectors $m_i \mathbf{e}_i$ are also period vectors.
But then, for any $b_1, \dots, b_d \in \Z$ if we let $b_i^\ast \in \{0, 1, \dots, m_i -1\}$
satisfy $b_i^\ast \equiv b_i \mod m_i$, then
$T^{(b_1, \dots, b_d)}\eta = T^{(b_1^\ast, \dots, b_d^\ast)}\eta$.
But then since $\prod_{i=1}^d \{0, 1, \dots, m_i - 1\}$ is finite, it follows
that the orbit of $\eta$ is finite.
\end{proof}
}

In higher dimensions, we will prove that for large $n$, the probability
that $X_\omega$ contains a coloring which is not periodic in $d$ linearly independent
directions tends to zero.
Throughout this section, we will keep the notation more manageable by
saying that a coloring $\eta \colon \Z^d \to \mathcal{A}$ has a certain property
on a set $B \subset \R^d$ if it has that property on $B \cap \Z^d$.
We will also refer to colorings of $B$ by which we mean colorings of
$B \cap \Z^d$.

We begin by proving that, with high probability, colorings in $X_\omega$ are 
locally periodic in $d$ linearly independent directions. 
The following lemma provides periodicity in one direction and illustrates the main idea 
of the more complicated general argument.

\begin{lemma}
\label{distinct}
Fix $d \in \N$.
For $\vre > 0$, 
let $D_\vre$ be the set of $\omega \in \Omega^d_n$ for which there is an 
$\omega$-legal coloring $\beta$ of $[1,n+2\vre n]^d$
such that the restrictions $\beta|_{\mathbf{x}+ [1+\vre n, n+\vre n]^d}$ are distinct for 
all integer vectors $\mathbf{x}$ with $\|\mathbf{x}\| < \frac{\vre n}{2}$.
Then for $\alpha < |\mathcal{A}|^{-(2d)^d\left(\frac{1+2\vre}{\vre}\right)^d}$, 
$\lim_{n\to \infty} \mu_{n,\alpha}(D_\vre) = 0$.
\end{lemma}

\begin{proof}
Let $A$ denote the set of $\beta \colon [1, n+2\vre n]^d \to \mathcal{A}$
such that the restrictions $\beta|_{\mathbf{x}+ [1+\vre n, n+\vre n]^d}$ are distinct for 
integer vectors $\mathbf{x}$ with $\|\mathbf{x}\| < \frac{\vre n}{2}$.
Note that there are at least $\frac{1}{(2d)^d}(\vre n)^d$
such vectors $\mathbf{x}$.
Given $\beta \in A$,
let $D_\vre(\beta)$ denote the set of $\omega \in \Omega$ such that $\beta$ is $\omega$-legal.
Since there are at least $\left(\frac{\vre n}{2d}\right)^d$ 
distinct colorings of the form $\beta|_{\mathbf{x}+ [1+\vre n, n+\vre n]^d}$,
all of which must be contained in $\omega$, we see that $\mu_{\alpha, n}(D_\vre(\beta)) \le \alpha^{(2d)^{-d}(\vre n)^d}$.
Since $D_\vre = \bigcup_{\beta \in A} D_\vre(\beta)$, we have
\begin{eqnarray*}
	\mu_{\alpha, n}(D_\vre) \le \sum_{\beta \in A} \mu_{\alpha, n}(D_\vre (\beta)) \le
			|A| \alpha^{(2d)^{-d}(\vre n)^d} &\le& 
				|\mathcal{A}^{[1,n+2\vre n]^d}| (\alpha^{(2d)^{-d}\vre^d})^{n^d} \\
			&=& 	(\alpha^{(2d)^{-d}\vre^d}|\mathcal{A}|^{(1+2\vre)^d})^{n^d}.
\end{eqnarray*}
If $\alpha < |\mathcal{A}|^{-(2d)^d\left(\frac{1+2\vre}{\vre}\right)^d}$, then the righthand side tends to $0$, which completes the proof.
\end{proof}

\ignore{
\begin{lemma}
\label{aperiodic-local}
Fix $\vre > 0$.
Given $\beta : [1,(1+\vre)n]^d \to \mathcal{A}$, let $m(\beta)$
be the maximal integer $m$ such that there exist linearly independent nonzero integer vectors
$\mathbf{p}_1, \dots, \mathbf{p}_m$ 
with $|\mathbf{p}_i| \le \vre n$ satisfying $\beta(\mathbf{j} + \mathbf{p}_i) = \beta(\mathbf{j})$
for all $\mathbf{j} \in [1+\vre n, n-\vre n]^d$ and $1 \le i \le m$.
Let $NP_p$ be the set of $\omega$ for which there exists an $\omega$-legal coloring $\beta : [1,n]^d \to \mathcal{A}$
with $m_p(\beta) < d$. Then for small enough $\alpha$, $\mu_\alpha(NP_{\vre n}) \to 0$ as $n \to \infty$.

Furthermore, for any $0 < \theta < \prod/2$ and for sufficiently small $\alpha > 0$, 
the $d$ period vectors can be chosen so that the angle between any two of them is at least $\theta$.
\end{lemma}
}

Note that this implies that, with probability tending to $1$,
every legal coloring of $[1 + 2\vre n, n]^d$ will have a period vector of length at most $\vre n$.
We want to extend this result and prove that, with probability tending to $1$, 
we can find $d$ period vectors 
which are small in magnitude and ``almost orthogonal.''
Before we do this, we first prove that the latter property implies linear independence.

\begin{lemma}
\label{lin-ind}
Let $\mathbf{p}_1, \dots, \mathbf{p}_m \in \R^d$ satisfy $|\mathbf{p}_i \cdot \mathbf{p}_j | \le \dfrac{1}{2m^2} \|\mathbf{p}_i\| \|\mathbf{p}_j\|$
for every $1 \le i < j \le m$. Then $\mathbf{p}_1, \dots, \mathbf{p}_m$ are linearly independent.
\end{lemma}

\begin{proof}
Let $\mathbf{p}_1, \dots, \mathbf{p}_k \in \R^d$ satisfy $|\mathbf{p}_i \cdot \mathbf{p}_j | \le \delta \|\mathbf{p}_i\| \|\mathbf{p}_j\|$
for every $1 \le i < j \le m$ and suppose they are linearly dependent. We will show that $\delta > \dfrac{1}{2m^2}$.
Linear dependence implies $\mathbf{p}_k = \sum_{i=1}^{k-1} a_i \mathbf{p}_i$
for some $k \le m$ and some $a_i \in \R$. Let $A = \displaystyle\max_{1 \le i \le k-1} \| a_i \mathbf{p}_i\|$.
Clearly, $\|\mathbf{p}_k\| < m A$, so we have
$$\| \mathbf{p}_k \|^2 =  \sum_{i=1}^{k-1} a_i \mathbf{p}_i \cdot \mathbf{p}_k  \le \sum_{i=1}^{k-1} |a_i| \delta \| \mathbf{p}_i \| \|\mathbf{p}_k\|
			< \delta m^2 A^2$$
But we also have
$$\| \mathbf{p}_k \|^2 = \left(   \sum_{i=1}^{k-1} a_i \mathbf{p}_i  \right) \cdot \left(   \sum_{i=1}^{k-1} a_i \mathbf{p}_i  \right) 
				= \sum_{i=1}^{k-1} a_i^2 \|\mathbf{p}_i\|^2 + \sum_{i\neq j} a_i a_j \mathbf{p}_i \cdot \mathbf{p}_j.$$
Combining these two inequalities we obtain
$$A^2 - \delta m^2 A^2 \le  \sum_{i=1}^{k-1} a_i^2 \|\mathbf{p}_i\|^2 + \sum_{i\neq j} a_i a_j \mathbf{p}_i \cdot \mathbf{p}_j
				= \|\mathbf{p}_k \|^2 < \delta m^2 A^2.$$
It follows that $1 < 2m^2 \delta$, as desired.
\end{proof}

\begin{lemma}
\label{aperiodic-local}
Fix $d \in \N$, $0 < \vre < 1/2$, and $0 < \delta < \dfrac{1}{2d^2}$.
Let $F_n = F_n(\vre, \delta)$ be the set of $\omega$ such that for every 
$\beta \in X_\omega$ and every $\mathbf{x} \in \Z^d$ there exist nonzero
integer vectors $\mathbf{p}_1, \dots, \mathbf{p}_d$ such that
\begin{enumerate}
\item $\|\mathbf{p}_i\| \le \vre n$ for all $1 \le i \le d$
\item 
\label{property2}
$|\mathbf{p}_i \cdot \mathbf{p}_j| \le \|\mathbf{p}_i\| \|\mathbf{p}_j\| \delta$ for all $1 \le i < j \le d$
\item $\beta(\mathbf{j} +\mathbf{p}_i) = \beta(\mathbf{j})$ for all 
$\mathbf{j} \in \mathbf{x} + [1+2\vre n, n]^d$ and $1 \le i \le d$.
\end{enumerate}
Then, there exists $\alpha_0 = \alpha(|\mathcal{A}|, d, \vre, \delta) > 0$
such that if $0 < \alpha < \alpha_0$,
$$\lim_{n \to \infty} \mu_{\alpha,n}(\Omega_n^d \setminus F_n) = 0.$$
\end{lemma}

\begin{proof}
\ignore{
For each $\mathbf{x} \in \Z^d$, let
$V(\mathbf{x})$ be the maximal volume of a box $B$ spanned by orthogonal vectors 
$\mathbf{p}_1,\dots, \mathbf{p}_d$ with the property that 
$\beta|_{[1,n]^d + \mathbf{j}}$ are distinct for $\mathbf{j} \in B$.
Note that by Lemma \ref{distinct}, $V(\mathbf{x}) \le \vre n^d$ for each $\mathbf{x} \in \Z^d$.
Fix $\mathbf{x}_0 \in \Z^d$ such that $V(\mathbf{x}_0) \ge V(\mathbf{x})$ for all $\mathbf{x}\in \Z^d$,
and let $\mathbf{p}_1, \dots, \mathbf{p}_d$ span $B$ as above. Suppose
$V(\mathbf{x}_0) \ge \vre n$.
Then there exists a legal coloring of $[1,n+2\vre n]^d$ such that....
But the probability of this event is at most
$$\alpha^{V(\mathbf{x}_0)}|\mathcal{A}|^{2^d V(\mathbf{x}_0)}.$$
For $\alpha < |\mathcal{A}|^{2^{-d}}$, this is less than
$(\alpha|\mathcal{A}|^{2^d})^{\vre n}$, which tends to $0$.
Hence, we must have $V(\mathbf{x}) < \vre n$ for every $\mathbf{x}\in \Z^d$.
Now fix $\mathbf{x} \in \Z^d$ and let $\mathbf{p}_1$ be the shortest vector 
such that $\beta|_{[1,n]^d + \mathbf{x}} = \beta|_{[1,n]^d + \mathbf{x} + \mathbf{p}_1}$.
Let $\mathbf{q}_1, \dots, \mathbf{q}_{d-1}$ be unit vectors orthogonal to $\mathbf{p}_1$.
Then the box spanned by $\delta \mathbf{p}_1$ and $\vre n \mathbf{q}_1, \dots, \vre n \mathbf{q}_{d-1}$
has volume at least $\vre^{d-1} n^{d-1} > \vre n$ (for large $n$) and therefore
there exist $\mathbf{j}_1$ and $\mathbf{j}_2$ in this box with the property that
$\beta|_{[1,n]^d + \mathbf{x} + \mathbf{j}_1} = \beta|_{[1,n]^d + \mathbf{x} + \mathbf{j}_2}$.
}

Let $\alpha_0 < \min\{|\mathcal{A}|^{-(\frac{4d^{5/2}}{\vre \delta})^d}, |\mathcal{A}|^{-(\frac{6d}{\delta})^d}\}$.
We begin by assigning a quantity $P(\beta)$ to each $\beta \in \mathcal{A}^{\Z^d}$.
For each $\mathbf{x} \in \Z^d$, let
$P(\mathbf{x})$ be the minimum value of $n^{d-k}\prod_{i=1}^k \|\mathbf{p}_i\|$,
taken over all sets of vectors $\{\mathbf{p}_1, \dots, \mathbf{p}_k\}$ which satisfy
\begin{enumerate}[label=(\alph*)]
\item $\|\mathbf{p}_1\| \le \|\mathbf{p}_2\| \le \dots \le \|\mathbf{p}_k\| < \vre n$\label{periodvectors1}
\item $\beta(\mathbf{x} + \mathbf{j} + \mathbf{p}_i) = \beta(\mathbf{x} + \mathbf{j})$
		for all $\mathbf{j} \in [1 + 2\vre n, n ]^d$ and all $1 \le i \le k$ \label{periodvectors2}
\item For each $1 \le j \le k$, $\mathbf{p}_j$ is a nonzero integer vector of minimal length among
		those satisfying
		the above two properties as well as
		$|\mathbf{p}_i \cdot \mathbf{p}_j| \le
		\delta \|\mathbf{p}_i\|^2 \le
		 \delta\|\mathbf{p}_i\| \|\mathbf{p}_j\|$
		for all $1 \le i < j$ \label{periodvectors3}
\end{enumerate}
If no such vectors exist, set $P(\mathbf{x}) = \infty$.
Note that for any $\omega \not\in D_\vre$ and any $\beta \in X_\omega$, $P(\mathbf{x})$ is
bounded above
by $\vre n^d$ for all $\mathbf{x} \in \Z^d$.
So by Lemma \ref{distinct} and our choice of $\alpha_0$, 
we may assume $P(\mathbf{x}) \le \vre n^d$ for each $\beta \in X_\omega$ and $\mathbf{x} \in \Z^d$.
Fix $\mathbf{x}_0 \in \Z^d$ such that $P(\beta) \df P(\mathbf{x}_0) \ge P(\mathbf{x})$ 
for all $\mathbf{x}\in \Z^d$,
and let $\mathbf{p}_1, \dots, \mathbf{p}_k$ be vectors satisfying \ref{periodvectors1}-\ref{periodvectors3}
with $P(\mathbf{x}_0) = n^{d-k}\prod_{i=1}^k \|\mathbf{p}_i\|$.
By Lemma \ref{lin-ind}, the vectors $\mathbf{p}_1, \dots, \mathbf{p}_k$ are linearly independent.
We claim that the probability that $P(\beta) \ge \vre n$ for some $\beta \in X_\omega$ tends to zero,
from which the lemma follows.
To show this, we consider two cases: $k < d$ and $k = d$.
We claim that, with probability tending to $1$, the first case doesn't occur, and the second
occurs only with $P(\mathbf{x}_0) < \vre n$.
\begin{enumerate}
\item[{\bf Case 1}:]\emph{$k < d$}

Let $\mathbf{q}_1, \dots \mathbf{q}_{d-k}$
be unit vectors orthogonal to the span of $\mathbf{p}_1, \dots, \mathbf{p}_k$
and orthogonal to each other,
and let $\mathbf{j}_0 = (\vre n,\vre n, \dots, \vre n)$.
We claim that the colorings $\beta|_{\mathbf{x_0} + [1, n]^d + \mathbf{j}_0 + \mathbf{j}}$
must be distinct for integer vectors $\mathbf{j} = \sum_{i=1}^k a_i \mathbf{p}_i + \sum_{i=1}^{d-k} b_i\mathbf{q}_i$,
where $0 \le a_i < \delta/d$ and $0 \le b_i < \frac{\vre n}{2d}$ are real numbers. Indeed, if not then there exist vectors
$\mathbf{j}_1$ and $\mathbf{j}_2$ of this form such that 
$$\beta|_{\mathbf{x}_0 + [1, n]^d + \mathbf{j}_0 + \mathbf{j}_1} 
		= \beta|_{\mathbf{x}_0 + [1,n]^d + \mathbf{j}_0 + \mathbf{j}_2}.$$
Letting $\mathbf{p}_{k+1} = \mathbf{j}_2 - \mathbf{j}_1$, it follows
that 
\begin{itemize}
\item $\|\mathbf{p}_{k+1}\| < d \frac{\delta}{d} \vre n + d \frac{\vre n}{2d} < \vre n$
\item $\beta(\mathbf{x}_0 + \mathbf{j} + \mathbf{p}_{k+1}) = \beta(\mathbf{x}_0 + \mathbf{j})$
for all $\mathbf{j} \in [1+ 2\vre n,n]^d$
\item $|\mathbf{p}_{k+1} \cdot \mathbf{p}_i| \le 
\sum_{j=1}^k a_i |\mathbf{p}_j \cdot \mathbf{p}_i| 
\le \frac{d\delta}{d} \|\mathbf{p}_i\|^2 = \delta\|\mathbf{p}_i\|^2$ for every $1 \le i \le k$.
\end{itemize}
Now, suppose $\|\mathbf{p}_i\| > \|\mathbf{p}_{k+1}\|$ for some $1 \le i \le k$
and let $i_0$ be the least such index.
Then
$|\mathbf{p}_i \cdot \mathbf{p}_{k+1}| \le \delta\|\mathbf{p}_i\|^2 \le \delta\|\mathbf{p}_i\|\|\mathbf{p}_{k+1}\|$
for all $ 1 \le i < i_0$, which violates the assumption that $\mathbf{p}_{i_0}$ is of minimal
length among such vectors. Thus,
$|\mathbf{p}_{k+1} \cdot \mathbf{p}_i| \le \delta\|\mathbf{p}_i\|^2 
\le \delta \|\mathbf{p}_{k+1}\|\|\mathbf{p}_i\|$ for all
$1 \le i \le k$, so (replacing $\mathbf{p}_{k+1}$ with a shorter vector
satisfying the above properties if necessary) $\{\mathbf{p}_1, \dots, \mathbf{p}_{k+1}\}$
satisfies conditions \ref{periodvectors1}-\ref{periodvectors3} above, violating the minimality of $n^{d-k}\prod_{i=1}^k \|\mathbf{p}_i\|$ in the definition of $P(\mathbf{x}_0)$.

But notice that, since the vectors $\mathbf{p}_1, \dots, \mathbf{p}_k, \mathbf{q}_1, \dots, \mathbf{q}_{d-k}$
are linearly independent, the number of integer vectors $\mathbf{j} = \sum_{i=1}^k a_i \mathbf{p}_i + \sum_{i=1}^{d-k} b_i\mathbf{q}_i$,
where $0 \le a_i < \delta/d$ and $0 \le b_i < \frac{\vre n}{2d}$, is at least 
$$\left(\displaystyle\prod_{i=1}^k (\delta/d) \|\mathbf{p}_i\| \right) \left( \displaystyle\prod_{i=1}^{d-k} \frac{\vre n}{2d} \right) = (\delta/d)^k (\vre/2d)^{d-k} P(\beta)$$
Thus, if the colorings $\beta|_{ \mathbf{x_0} + [1, n]^d + \mathbf{j_0} + \mathbf{j}}$
are distinct for $\mathbf{j} = \sum_{i=1}^k a_i \mathbf{p}_i + \sum_{i=1}^{d-k} b_i\mathbf{q}_i$
($0 \le a_i < \frac{\delta}{d}, 0 \le b_i < \frac{\vre n}{2d}$),
then by the maximality of $P(\mathbf{x}_0)$ there is an $\omega$-legal coloring of $[1, n + 2\vre n]^d$
such that
\begin{itemize}
\item $[1, n + 2\vre n]^d$ is the union of $2^d$ translates of $[1+2\vre n, n]^d$,
		each of which having period vectors $\mathbf{p}_1', \dots, \mathbf{p}_{k'}'$
			with $\|\mathbf{p}_i'\| < \vre n$
			satisfying $n^{d-k'}\prod_{i=1}^{k'} \|\mathbf{p}_i'\| \le P(\beta)$
\item $[1, n + 2\vre n]^d$ contains at least $(\frac{\delta}{d})^d (\frac{\vre}{2d})^d P(\beta)$ translates of
		$[1,n]^d$ on which the restrictions of this legal coloring are distinct.
\end{itemize}
Now, there are most $((2\vre n)^d + 1)^d \le (3\vre n)^{d^2}$ choices of a set of at most $d$ integer vectors
with norm at most $\vre n$. And once these period vectors are fixed, there
are most $|\mathcal{A}|^{(\sqrt{d}n)^{d-k'}\prod_{i=1}^{k'} \|\mathbf{p}_i'\|}$ colorings
of $[1 + 2\vre n, n]^d$ that are periodic with those period vectors. Hence, the first property
bounds the number of such colorings of $[1, n + 2\vre n]^d$ by 
$(3\vre n)^{d^22^d} |\mathcal{A}|^{(2\sqrt{d})^dP(\beta)}$. 
The second property implies that the probability of any one of
these being $\omega$-legal is at most $\alpha^{(\frac{\vre\delta}{2d^2})^d P(\beta)}$.
So the probability that such a coloring is $\omega$-legal, and hence the probability that 
Case 1 occurs, is at most 
$$(3\vre n)^{d^22^d}|\mathcal{A}|^{(2\sqrt{d})^d P(\beta)}\alpha^{(\frac{\vre \delta}{2d^2})^{d} P(\beta)}.$$
Since every integer vector has length at least $1$ and $k < n$,
$P(\beta) \ge n$.
Since $\alpha < |\mathcal{A}|^{-(\frac{4d^{5/2}}{\vre \delta})^d}$, we have
$$(3\vre n)^{d^22^d}|\mathcal{A}|^{(2\sqrt{d})^d P(\beta)}\alpha^{(\frac{\vre \delta}{2d^2})^{d} P(\beta)}
\le (3\vre n)^{d^22^d}(\alpha^{(\frac{\vre \delta}{2d^2})^{d}}|\mathcal{A}|^{(2\sqrt{d})^d})^{n},$$
which tends to zero.

\item[{\bf Case 2}:]\emph{$k = d$}

Again, by the maximality of $P(\mathbf{x_0})$, there are at most 
$(3\vre n)^{d^22^d}|\mathcal{A}|^{(2\sqrt{d})^d P(\mathbf{x}_0)}$
possible colorings of $\mathbf{x}_0  + [1, n + 2\vre n]^d$.
Consider the colorings $\beta|_{\mathbf{x_0} + [1, n]^d + \mathbf{j_0} + \mathbf{j}}$
for $\mathbf{j} = \sum_{i=1}^d a_i \mathbf{p}_i$, where $0 \le a_i < \frac{\delta}{3d}$
for $1 \le i < d$ and $0 \le a_d \le \frac13$.
If these colorings are not all distinct, 
say $\beta|_{\mathbf{x_0} + [1, n]^d + \mathbf{j_0} + \mathbf{j}_1} = \beta|_{\mathbf{x_0} + [1, n]^d + \mathbf{j_0} + \mathbf{j}_2}$,
then letting $\mathbf{p}_d' = \mathbf{j}_2 - \mathbf{j}_1$
we have
$$\|\mathbf{p}_d'\| < \sum_{i=1}^{d-1} \frac{\delta}{3d}\|\mathbf{p}_i\| + \frac{1}{3}\|\mathbf{p}_d\| 
		\le  d\frac{\delta}{3d}\|\mathbf{p}_d\| + \frac{1}{3}\|\mathbf{p}_d\| < \|\mathbf{p}_d\|.$$
Furthermore, since $\|\mathbf{j}_1\| \le \vre n$ and 
$$\beta(\mathbf{x}_0 + \mathbf{j}_0 + \mathbf{j_1} + \mathbf{j}) = \beta(\mathbf{x}_0 + \mathbf{j}_0 + \mathbf{j_2} + \mathbf{j})
											=  \beta(\mathbf{x}_0 + \mathbf{j}_0 + \mathbf{j_1} + \mathbf{j} + \mathbf{p}_d' )$$
for every $\mathbf{j} \in [1,n]^d$, 		
$\mathbf{p}_d'$ satisfies property \ref{periodvectors2} above as well.
Finally, we have
$$|\mathbf{p}_d' \cdot \mathbf{p}_i| \le \frac{\delta}{3d}\|\mathbf{p}_i\|^2
		+ d\frac{\delta}{3d}\delta\|\mathbf{p}_i\|^2 + \frac{1}{3}\delta\|\mathbf{p}_i\|^2
		\le \delta \|\mathbf{p}_i\|^2.$$
Now, if $\|\mathbf{p}_d'\| < \|\mathbf{p}_i\|$ for some $1 \le i < d$, then, arguing
as in Case 1, the minimality in \ref{periodvectors3} would be violated
for some $1 \le j < d$, so we must have 
$|\mathbf{p}_d' \cdot \mathbf{p}_i| \le \delta \|\mathbf{p}_i\|^2 \le \delta \|\mathbf{p}_i\|\|\mathbf{p}_d'\|$, violating the minimality of $\mathbf{p}_d$.

Hence, the colorings $\beta|_{\mathbf{x_0} + [1, n]^d + \mathbf{j_0} + \mathbf{j}}$ must be distinct 
for $\mathbf{j} = \sum_{i=1}^k a_i \mathbf{p}_i$.
But then we have a legal coloring
of $[1, n + 2\vre n]^d$ such that $\frac{\delta^d}{3^dd^d}P(\mathbf{x}_0)$ 
different translates of $[1,n]^d$ within
$\mathbf{x}_0  + [1, n + 2\vre n]^d$ are colored differently.
The probability of this event is at most
$$(3\vre n)^{d^22^d}|\mathcal{A}|^{(2\sqrt{d})^dP(\mathbf{x}_0)}
					\alpha^{\frac{\delta^{d} }{3^dd^d} P(\mathbf{x}_0)}.$$
If $\alpha < |\mathcal{A}|^{(\frac{6d^{3/2}}{\delta})^d}$ and 
$P(\mathbf{x}_0) \ge \vre n$, then this is at most
$$(3\vre n)^{d^22^d}(|\mathcal{A}|^{(2\sqrt{d})^d}\alpha^{\frac{\delta^{d} }{3^dd^d}})^{\vre n},$$ 
which tends to zero.
By our choice of $\alpha$, it follows that with probability tending to $1$, $P(\beta) < \vre n$,
and therefore $P(\mathbf{x}) < \vre n$ for all $\mathbf{x} \in \Z^d$.
The lemma then follows from the definition
of $P$, since if
$\{\mathbf{p}_1', \dots, \mathbf{p}_{k'}'\}$ are the period vectors in the definition of $P(\mathbf{x})$,
we have $\mathbf{p}_i' \in \Z^d$ which implies
$\vre n > n^{d-k'}\prod_{i=1}^{k'} \|\mathbf{p}_i'\| \ge n^{d-k'}$ and hence $k'=d$.\qedhere
\end{enumerate}
\end{proof}

The proof of the general case of Theorem \ref{aperiodic} consists of a ``local-to-global'' argument extending the local periodicity
guaranteed in the previous lemma to all of $\Z^d$.

\begin{proof}[Proof of Theorem \ref{aperiodic}]
Fix $ \vre < (2d)^{-5/2}$ and $\delta = \frac{1}{2d^2}$. Let
$\alpha_0 = \alpha_0(|\mathcal{A}|, d, \vre, \delta)$ be as in Lemma \ref{aperiodic-local}
and let $0 < \alpha < \alpha_0$.
(Note that $\vre$ and $\delta$ both depend only on $d$, so $\alpha_0$ is
a function of $|\mathcal{A}|$ and $d$ as in the statement of the theorem.)
By Lemma \ref{aperiodic-local} and Remark \ref{bounded-period=finite}, it will suffice to show that
whenever $\omega \in F_n(\vre, \delta)$, every orbit in $X_\omega$ is periodic in each cardinal direction
with period at most $(2^dn)^{2^d}$.
Let $\omega \in F_n(\vre, \delta)$ and let $\eta \in X_\omega$.
Since $\omega \in F_n(\vre, \delta)$, by Lemma \ref{lin-ind} there exist linearly independent
integer vectors $\mathbf{p}_1, \dots, \mathbf{p}_d$ with $\|\mathbf{p}_i\| \le \vre n$
such that for any $1 \le i \le d$ and any $\mathbf{j} \in B \df [1+2\vre n, n]^d$, 
$\eta(\mathbf{j} + \mathbf{p}_i) = \eta(\mathbf{j})$.
Suppose for contradiction that $\eta$ is not periodic with the same period vectors on
$B + \mathbb{R}_{\ge 0}\mathbf{p}_1$.
Then there exists a minimal $t \in [0,\infty)$ such that $\eta(\mathbf{j} -\mathbf{p}_i) \neq \eta(\mathbf{j})$
for some integer vector $\mathbf{j} \in B + t\mathbf{p}_1$.
\ignore{Let $B' = B + t \mathbf{p}_1 + \frac{n}{2\|\mathbf{p}_1\|} \mathbf{p}_1$.
If $B(\mathbf{j}_0, 2\vre n) \subset B'$, then Lemma \ref{aperiodic-local}
guarantees a nonzero period vector $\mathbf{p}'$ for $\eta$ on $B'$ and either $\mathbf{j}_0 + \mathbf{p}'$
or $\mathbf{j}_0 - \mathbf{p}'$ lies in the set $B + t'\mathbf{p}_1$ for some $t'<t$. Without loss of generality, we
may assume the former, so we have
that $$\eta(\mathbf{j}_0) = \eta(\mathbf{j}_0 + \mathbf{p}') = \eta(\mathbf{j}_0 + \mathbf{p}' - \mathbf{p}_1)
		= \eta(\mathbf{j}_0 - \mathbf{p}_1),$$
which is a contradiction.
If $B(\mathbf{j}_0, 2\vre n) \not\subset B'$,
then again using}
Let $\mathbf{j}_0 \in (B + t\mathbf{p}_1) \setminus \cup_{t' < t} (B + t'\mathbf{p}_1)$
 satisfy
$\eta(\mathbf{j}_0 -\mathbf{p}_i) \neq \eta(\mathbf{j}_0)$,
and let $B'$ be an integer translate of $[1+3\vre n, n - \vre n]^d$ 
with center at most distance $\sqrt{d}$ from $\mathbf{j}_0$.
We may assume $n$ is large enough that $\sqrt{d} < \frac12 n$.
Since $\delta < \frac{1}{2(d-1)}$, there exist 
linearly independent integer vectors $\mathbf{p}_1', \dots, \mathbf{p}_d'$ with $\|\mathbf{p}_i'\| \le \vre n$
such that 
$$\eta(\mathbf{j} + \mathbf{p}_i') = \eta(\mathbf{j} - \mathbf{p}_i') = \eta(\mathbf{j})$$
for every $\mathbf{j} \in B'$ and
\begin{equation}
\label{almost-orthogonal}
|\mathbf{p}_i'\cdot \mathbf{p}_j'| \le \dfrac{\|\mathbf{p}_i'\|\|\mathbf{p}_j'\|}{2(d-1)}\text{ for }1 \le i < j \le d.
\end{equation}
Indeed, Lemma \ref{aperiodic-local} asserts that we may find such vectors which satisfy 
$\eta(\mathbf{j}) = \eta(\mathbf{j} + \mathbf{p}_i')$
on a given translate of $[1 + 2\vre n, n]^d$, which implies
$\eta(\mathbf{j} + \mathbf{p}_i') = \eta(\mathbf{j} - \mathbf{p}_i') = \eta(\mathbf{j})$
on the corresponding translate of $[1+3\vre n, n - \vre n]^d$.
We claim there exists a sequence of indices $1 \le i_j\le d$ 
and integers $\sigma_j \in \{-1,1\}$ for $1 \le j \le J$
 such that
\begin{enumerate}
	\item  For each $1 \le \ell \le J$, 
	$\mathbf{j}_0 + \displaystyle\sum_{j=1}^{\ell} \sigma_j\mathbf{p}_{i_j}' \in B'$ 
			and $\mathbf{j}_0 - \mathbf{p}_1 + 
			\displaystyle\sum_{j=1}^{\ell} \sigma_j\mathbf{p}_{i_j}' \in B'$ \label{chain1}
	\item $\mathbf{j}_0 + \displaystyle\sum_{j=1}^{J} \sigma_j\mathbf{p}_{i_j}' \in B + t'\mathbf{p}_1$ 
	for some $t' < t$. \label{chain2}
\end{enumerate}
If we prove this, then we will have
$$\eta(\mathbf{j_0}) = \eta\left(\mathbf{j}_0 + \displaystyle\sum_{j=1}^{J} \sigma_j\mathbf{p}_{i_j}'\right)
				= \eta\left(\mathbf{j}_0 + \displaystyle\sum_{j=1}^{J} \sigma_j\mathbf{p}_{i_j}' - \mathbf{p}_1\right)
				= \eta(\mathbf{j}_0 - \mathbf{p}_1),$$
which is a contradiction.
To prove the claim, 
let $\mathbf{y}'$ be the center of $B + t\mathbf{p}_1$ and let
$\mathbf{y} = \mathbf{j}_0 + 2d^2\vre n\dfrac{\mathbf{y}' - \mathbf{j}_0}{\|\mathbf{y}' - \mathbf{j}_0\|}$.
Note that the distance from $\mathbf{y}$ to the boundary of $B + t\mathbf{p}_1$
is at most $\frac{2d^2\vre n}{\sqrt{d}} = 2d^{3/2}\vre n$
and hence
\begin{equation}
\label{dist-y-to-boundary}
B(\mathbf{y}, 2d^{3/2}\vre n) \subset B + t\mathbf{p}_1.
\end{equation}
There exist $x_1, \dots, x_d \in \R$ such that 
$\mathbf{j}_0 + \displaystyle\sum_{i=1}^d x_i \mathbf{p}_i' = \mathbf{y}$. 
We will choose
our indices so that $i_j = i$ for approximately $|x_i|$ values of $j$, but to ensure that 
(\ref{chain1}) holds we need to show that the coefficients $x_i$ are small relative to the 
distance $\|\mathbf{y} - \mathbf{j}_0\| = 2d^2\vre n$.
Note that 
$$\left\|\displaystyle\sum_{i=1}^d x_i \mathbf{p}_i' \right\|^2 = \displaystyle\sum_{i=1}^d x_i^2 \|\mathbf{p}_i'\|^2
+ \displaystyle\sum_{i\neq j} (x_ix_j \mathbf{p}_i' \cdot \mathbf{p}_j').$$
By (\ref{almost-orthogonal}),
we have that $|x_ix_j \mathbf{p}_i' \cdot \mathbf{p}_j'| \le \frac{|x_i||x_j|\|\mathbf{p}_i'\|\|\mathbf{p}_j'\|}{2(d-1)}$,
so $$\|\mathbf{y} - \mathbf{j}_0\|^2 \ge 
 \displaystyle\sum_{i=1}^d x_i^2 \|\mathbf{p}_i'\|^2 - \frac{1}{2(d-1)}\displaystyle\sum_{i \neq j} (|x_i|\|\mathbf{p}_i'\|)(|x_j|\|\mathbf{p}_j'\|).$$
Using Lagrange multipliers it is easy to see that,
subject to the constraint that the right-hand side is at most $2d^2 \vre n$,
$\sum_{i = 1}^d |x_i|\|\mathbf{p}_i'\|$ is maximized when 
$$|x_1|\|\mathbf{p}_1'\| = |x_2|\|\mathbf{p}_2'\|= \dots = |x_d|\|\mathbf{p}_d'\| \df x,$$ 
in which case we have
$$2d^2\vre n \ge \sqrt{dx^2 - \frac{d}{2}x^2} = \sqrt{\frac{d}{2}} x = \frac{d}{\sqrt{2d}} x = 
\frac{1}{\sqrt{2d}} \sum_{i=1}^d (|x_i|\|\mathbf{p}_i'\|).$$
Now define $\mathbf{z} = \displaystyle\sum_{i=1}^d \overline{x_i} \mathbf{p}_i'$,
where $\overline{x_i} = \lfloor{x_i}\rfloor$ if $x_i \ge 0$ and $\overline{x_i} = \lceil{x_i}\rceil$ if $x_i \le 0$.
Let $J = \sum_{i=1}^d |\overline{x_i}|$ and for 
$\sum_{i=1}^{m-1} |\overline{x_i}| < j \le \sum_{i=1}^{m} |\overline{x_i}|$
set $i_j = m$ and $\sigma_j = \frac{\overline{x_m}}{|\overline{x_m}|}$.
Then $\mathbf{z} = \sum_{i=1}^J \sigma_j \mathbf{p}_{i_j}'$ and
$\|\mathbf{z} - \mathbf{y}\| \le \sum_{i=1}^d \|\mathbf{p}_i'\| \le d \vre n$,
so by (\ref{dist-y-to-boundary}) $\mathbf{z}$ is contained in the interior of $B + t\mathbf{p}_1$,
 so (\ref{chain2}) holds.
Also, for each $1 \le \ell \le J$,
$$\left\|\mathbf{j}_0 - \sum_{i=1}^\ell \sigma_j \mathbf{p}_{i_j}'\right\| \le \sum_{i=1}^d |\overline{x_i}| \|\mathbf{p}_i'\| \le \sum_{i=1}^d |{x_i}| \|\mathbf{p}_i'\| \le 
2^{3/2}d^{5/2}\vre n.$$
Since we assume
$2^{3/2}d^{5/2} \vre < 1/2$, 
we have that the distance from $\sum_{i=1}^\ell \sigma_j \mathbf{p}_{i_j}'$
to the center of $B'$ is at most $\frac12n + \sqrt{d} < \frac12n + \frac12n = n$,
so (\ref{chain1}) follows
and the proof of the claim is complete.
Hence, $\eta$ is periodic with the same period vectors on $B + \R_{\ge 0}\mathbf{p}_1$.
Using the same argument, we see that $\eta$ is periodic with period vectors 
$\mathbf{p}_1,\dots, \mathbf{p}_d$
on $B + \R\mathbf{p}_1 + \R\mathbf{p}_2 + \dots + \R\mathbf{p}_d$
and hence on all of $\Z^d$.

Now, if $\mathbf{p}_i = (p^{(0)}_{i,1}, \dots, p^{(0)}_{i,d}) \in \Z^d$,
then for each $1 \le i \le d-1$, let
$$\mathbf{p}_i^{(1)} = (p^{(1)}_{i,1}, \dots, p^{(1)}_{i,d})
		\df  p^{(0)}_{i,1}\mathbf{p}_d - p^{(0)}_{d,1}\mathbf{p}_i.$$
We obtain $d-1$ period vectors with $\mathbf{e}_1$-component equal to zero
and $\|\mathbf{p}_i^{(1)} \| \le 2(\vre n)^2 \le (2\vre n)^2$.
Now suppose for $1 \le k \le d-1$ we have $d-k$ period vectors for $\eta$ satisfying
$\|\mathbf{p}_i^{(k)}\| \le (2^{k}\vre n)^{2^{k}}$ and $\mathbf{e}_i \cdot \mathbf{p}_i^{(k)} = 0$
for $1 \le i \le k$.
For $1 \le i \le d - (k+1)$, 
define $$\mathbf{p}_i^{(k+1)} = (p^{(k+1)}_{i,1}, \dots, p^{(k+1)}_{i,d})
			\df  p^{(k)}_{i,k+1}\mathbf{p}_{d-k}^{(k)} - p^{(k)}_{d-k,k+1}\mathbf{p}_i^{(k)}.$$
These are $d- (k+1)$ period vectors with 
$\|\mathbf{p}_i^{(k+1)}\| \le 2(2^k\vre n)^{2^{k+1}} \le (2^{k+1}\vre n)^{2^{k+1}}$ and 
$\mathbf{e}_i \cdot \mathbf{p}_i^{(k+1)} = 0$
for $1 \le i \le k+1$. By finite induction, we obtain a period vector parallel to $\mathbf{e}_d$
with period at most $(2^d\vre n)^{2^{d}}$. Arguing similarly, we can produce a period vector in
each of the cardinal directions with period at most $(2^d\vre n)^{2^{d}}$.
By Remark \ref{bounded-period=finite}, this completes the proof.
\end{proof}

\section{Open questions}
\label{entropy}

In \cite{MP}, it is shown that if $\alpha < 1/|\mathcal{A}|$, then for each $\vre > 0$
the probability that the entropy of $X_\omega$ is at least $\vre$ tends to zero as $n$ tends to infinity.
However it is still possible that the probability of $X_\omega$ having zero entropy tends to zero
as well. Even if the entropy is generically zero, this leaves open questions about directional entropy
and periodicity. As with the definition of entropy given above,
we may also define directional entropy using complexity.

\begin{defn}
If $\omega \in \Omega_n^d$; $m \in \{1, 2, \dots, d-1\}$;
$\mathbf{u}_1, \dots, \mathbf{u}_m$ are linearly independent unit vectors in $\R^d$;
and $V = \spn(\mathbf{u}_1, \dots, \mathbf{u}_m)$, set
$$R_{V,k,t} = \left\{\sum_{i=1}^m a_i \mathbf{u}_i + \sum_{i=m+1}^d b_i \mathbf{u}_i
			\colon a_i \in [0,k], b_i \in [0,t]\right\},$$
where $\mathbf{u}_i$ ($m+1\le i \le d$) are unit vectors orthogonal to $V$ which complete
a basis for $\R^d$. The $m$-dimension (topological) directional entropy in direction $V$ is
$$h_V(X_\omega) = \sup_{t > 0} \lim_{k \to \infty} \dfrac{\log(P_\omega(R_{V,k,t} \cap \Z^d))}{k^m}.$$
\end{defn}

Again, it is straightforward to show that this definition coincides 
with the general definition of topological directional entropy given in \cite{Mi}.
If $h(X_\omega) > 0$, then $h_V(X_\omega) = \infty$ for all proper subspaces 
$V \subset \R^d$. \ignore{To see this, suppose $h_V(X_\omega) \le \log a$ for some $a > 1$. Then
for each $t > 0$ there exists $k_t$ such that for $k \ge k_t$, we
have $P_\omega(R_{V, k, t} \cap \Z^d) \le a^{k^m}$. But $[1,k]^d \cap \Z^d$
can be written as a union of $c\left(\frac{k}{t}\right)^{d-m}$ translates of $R_{V, k, t}\cap \Z^d$,
where $c > 0$ is independent of $k$ and $t$. Hence,
$P_\omega(k) \le (a^{k^m})^{ck^{d-m}/t^{d-m}}  = a^{ck^d/t^{d-m}}$ 
for all sufficiently large $k$ and hence,
$h(X_\omega) \le \frac{c \log a}{t^{d-m}}$. Since $t$ was arbitrary,
$h(X_\omega) = 0$.} But if $h(X_\omega) = 0$, 
the directional entropy may be zero, positive, or even infinite in any given direction.
Ledrappier's three-dot system (see \cite{L}) provides an example of a zero-entropy SFT with positive but finite directional entropy in all directions. 
The system corresponding to $\omega = \{\beta\} \in \Omega_2^2$, where $\beta(i,j) = 1$ for $1 \le i,j \le 2$
(i.e. the SFT containing only a single constant coloring of $\Z^2$) is of course an example
where the directional entropy is zero in all directions.
The following example shows that it is possible to have infinite entropy in all directions as well.

\begin{example}
Set $\mathcal{A} = \{0, 1, 2, 3\}$. Let
$$\omega_1 = \{ \eta \in \{0,1\}^{[1,2]^2} \colon \eta(i,1) = 0 \Rightarrow \eta(i,2) = 0\},$$
$$\omega_2 = \{ \eta \in \{2, 3\}^{[1,2]^2} \colon \eta(1,i) = 2 \Rightarrow \eta(2, i) = 2\},$$
and $\omega = \omega_1 \cup \omega_2$.
Then $X_\omega$ has zero entropy but has infinite directional entropy in all directions.
\end{example}

\begin{proof}
To prove both claims, let us find $P_\omega(k, w)$ for arbitrary $k, w \in \N$. 
Note that for any $\eta \in X_\omega$, either $\im \eta \subset \{0,1\}$ or $\im \eta \subset \{2,3\}$.
We first consider the former case.
Then for any $(i,j) \in \Z^2$, if $\eta(i,j) = 0$ then $\eta(i,j') = 0$ for all 
$j'\ge j$. Of course it follows also that if $\eta(i,j) = 1$ then $\eta(i,j') = 1$ for all 
$j'\le j$. Hence $\eta |_{\{i\} \times [1,w]}$ is completely determined by 
$\min\{1 \le j \le w : u(i,j) = 0\}$, where this min is taken to be $0$ if the set is empty.
Furthermore, if $\alpha_i \colon \{1,\dots,w\} \to \{0,1\}$ satisfies, for some $j_i \in \{1, \dots, w+1\}$,
$\alpha_i(j) = 0$ for $j \ge j_i$ and $\alpha_i(j) = 1$ for $j \le j_i - 1$,
then it is clear from the definition of $\omega$ that
$\alpha(i,j) = \alpha_i(j)$ is an $\omega$-legal coloring of $[1,k]\times [1,w]$.
Similarly, if $\im \eta = \{2, 3\}$ then $\eta |_{[1,k]\times[1,w]}$ is determined by
$\min\{1 \le i \le k : \eta(i,j) = 2\}$ for $1\le j \le w$. Thus,
$P_\omega(k, w) = (w+1)^k + (k+1)^w$.
From this it follows that
$$h(X_\omega) = \lim_{k \to \infty} \dfrac{\log(P_\omega(k, k))}{k^2} = 
	 \lim_{k \to \infty} \dfrac{\log(2(k+1)^k)}{k^2}
	 =  \lim_{k \to \infty} \left[\dfrac{\log 2}{k^2} + \dfrac{\log(k+1)}{k}\right] = 0.$$
We also obtain
$$h_{\mathbf{e}_1}(X_\omega) = \sup_{w \in \N} \lim_{k \to \infty} \dfrac{\log(P_\omega(k, w))}{k} 
			\ge \sup_{w\in \N} \lim_{k \to \infty} \dfrac{\log(w+1)^k}{k}
			= \sup_{w\in\N} \log(w+1) = \infty$$
and similarly
$$h_{\mathbf{e}_2}(X_\omega) = \sup_{w \in \N} \lim_{k \to \infty} \dfrac{\log(P_\omega(w, k))}{k} 
			\ge \sup_{w\in \N} \lim_{k \to \infty} \dfrac{\log(w+1)^k}{k}
			= \sup_{w\in\N} \log(w+1) = \infty.$$
If $\mathbf{v}$ is some unit vector not parallel to $\mathbf{e}_2$, then there
is a constant $c > 0$ such that if $\mathcal{L}_k$
is the line segment connecting $-\frac{k}{2}\mathbf{v}$ to $\frac{k}{2}\mathbf{v}$, we have
that $\mathcal{L}_k^{(w)} \cap \Z^2$ consists of at least $ck$ vertical lines of length at least $w$
and hence
$$h_{\mathbf{v}}(X_\omega) = \sup_{w \in \N} \lim_{k \to \infty} 
	\dfrac{\log(P_\omega(\mathcal{L}_k^{(w)}))}{k} \ge 
	\sup_{w \in \N} \lim_{k \to \infty} \dfrac{\log((w+1)^{ck})}{k} 
	= \sup_{w\in \N} c\log(w+1) = \infty.$$
\end{proof}

Fix $\mathcal{A}$ and let $\alpha_d = \alpha(d, |\mathcal{A}|)$ be as in 
Theorem~\ref{aperiodic}. We ask the following.

\begin{question}
For $\alpha_d \le \alpha < \frac{1}{|\mathcal{A}|}$, what is
\begin{itemize}
	\item $\displaystyle\lim_{n\to \infty} \mu_{\alpha, n}\{\omega \colon h(X_\omega) = 0\}$?
	\item $\displaystyle\lim_{n\to \infty} \mu_{\alpha, n}\{\omega \colon h_V(X_\omega) =
			 0\text{ for some }V\}$?
	\item $\displaystyle\lim_{n\to \infty} \mu_{\alpha, n}\{\omega \colon h_V(X_\omega) 
				= 0\text{ for all }V\}$?
	\item $\displaystyle\lim_{n\to \infty} \mu_{\alpha, n}\{\omega \colon h_V(X_\omega) 
			= \infty\text{ for some }V\}$?
	\item $\displaystyle\lim_{n\to \infty} \mu_{\alpha, n}\{\omega \colon h_V(X_\omega) 
			= \infty\text{ for all }V\}$?
	\item $\displaystyle\lim_{n\to \infty} \mu_{\alpha, n}\{\omega \colon |X_\omega| < \infty\}$?
	\item $\displaystyle\lim_{n\to \infty} \mu_{\alpha, n}\{\omega \colon 
				X_\omega\text{ does not contain any infnite orbits}\}$?
	\item $\displaystyle\lim_{n\to \infty} \mu_{\alpha, n}\{\omega \colon\text{ every element of } 
				X_\omega\text{ has at least one period vector}\}$?
\end{itemize}
\end{question}

Another natural question concerns the critical value of the parameter $\alpha_0 = 1/|\mathcal{A}|$ with $d = 1$.
Neither our result nor the results in \cite{M} include this case, so it is natural to ask the following.

\begin{question}
For $d = 1$ and $\alpha = \frac{1}{|\mathcal{A}|}$, what is
\begin{itemize}
	\item $\displaystyle\lim_{n\to \infty} \mu_{\alpha, n}\{\omega \colon h(X_\omega) = 0\}$?
	\item $\displaystyle\lim_{n\to \infty} \mu_{\alpha, n}\{\omega \colon |X_\omega| < \infty\}$?
\end{itemize}
\end{question}
Note that in several places in the proof of the $d=1$ case of Theorem \ref{aperiodic}, 
we use the fact that $(\alpha |\mathcal{A}|)^n$ (and even $n (\alpha |\mathcal{A}|)^n$)
approaches $0$ as $n \to \infty$, so the present methods do not seem to yield any information in the case $\alpha = 1/|\mathcal{A}|$.

\section*{Acknowledgements} The author thanks the anonymous referee for several helpful comments.

\end{document}